\documentclass[12pt]{article}
\usepackage{fullpage,comment,authblk,stackrel}
\usepackage{nicefrac,color}
\usepackage{hyperref}
\hypersetup
{
    colorlinks,	%
    citecolor=blue,%
    filecolor=black,%
    linkcolor=blue,%
    urlcolor=blue
}
\usepackage{amssymb,amsmath,amsthm}
\usepackage{eulervm, xfrac}
\usepackage{amscd}
\usepackage{mathrsfs}
\usepackage{rotating, setspace}
\DeclareMathAlphabet{\mathpzc}{OT1}{pzc}{m}{it}
\usepackage[all, 2cell]{xy}
\UseTwocells

\renewcommand{\sfdefault}{iwona}
\DeclareMathAlphabet{\mathbfsf}{\encodingdefault}{\sfdefault}{bx}{n}

\newcommand{\define}[1]{\textbf{#1}}

\makeatletter
\def\bign#1{\mathclose{\hbox{$\left#1\vbox to8.5\p@{}\right.\n@space$}}\mathopen{}}
\def\Bign#1{\mathclose{\hbox{$\left#1\vbox to11.5\p@{}\right.\n@space$}}\mathopen{}}
\def\biggn#1{\mathclose{\hbox{$\left#1\vbox to14.5\p@{}\right.\n@space$}}\mathopen{}}
\def\Biggn#1{\mathclose{\hbox{$\left#1\vbox to17.5\p@{}\right.\n@space$}}\mathopen{}}
\makeatother

\newcommand{\Y}				{{\mathbb Y}}
\newcommand{\Xspace}        {{\mathbb X}}

\newcommand{\Rspace}        	{{\mathbb R}}

\newcommand{\Vect}          	{{\mathsf{Vec}}}

\newcommand{\Ent}          		{{\mathsf{Ent}}}

\newcommand{\Bfunc}          	{{\mathsf{B}}}
\newcommand{\Cfunc}          	{{\mathsf{C}}}

\newcommand{\Ffunc}          	{{\mathsf{F}}}

\newcommand{\Hfunc}          	{{\mathcal{H}}}

\newcommand{\Lfunc}          	{{\mathsf{L}}}

\newcommand{\image}		{\mathrm{im}}

\newcommand{\X}			{\mathsf{X}}
\newcommand{\B}			{\mathsf{B}}
\newcommand{\Z}			{\mathsf{Z}}
\newcommand{\E}			{\mathsf{E}}
\newcommand{\D}			{\mathsf{D}}

\newcommand{\field}			{\mathsf{k}}

\newtheoremstyle{amit}
{7pt}
{7pt}
{}
{7pt}
{\bf}
{:}
{.5em}
{}

\usepackage{tikz-cd}
\usepackage[font=small]{caption}

\hypersetup{
  colorlinks   = true, 
  urlcolor     = blue, 
  linkcolor    = blue, 
  citecolor   = blue 
}

\theoremstyle{amit}
\newtheorem{defn}{Definition}[section]
\newtheorem{prop}[defn]{Proposition}
\newtheorem{lem}[defn]{Lemma}
\newtheorem{thm}[defn]{Theorem}
\newtheorem{cor}[defn]{Corollary}

\title{Leray Spectral Sequence for Simplicial Maps}
\author{Amit Patel and Dustin Sauriol 
\thanks{This research was partially supported by NSF grant CCF 1717159}}
\affil{Department of Mathematics, Colorado State University}
\date{}

\begin{document}

\maketitle
 
 \begin{abstract}
The Leray spectral sequence of a map $f$ computes the homology
of the domain of $f$ from the fibers of $f$.
In this expository paper, we relate in full detail the Leray spectral sequence associated to a
simplicial map $f$ to the Leray cosheaves of $f$.
We then give applications to level set persistent homology and Reeb spaces.
 \end{abstract}

\section{Introduction}

In homological algebra, a spectral sequence is a divide-and-conquer approach to
computing the homology of a space.
Divide the space into subspaces and calculate the homology of each subspace in parallel.
Combine the subspaces in a hierarchical fashion updating the homology at each level.
Assuming a finite number of subspaces,
this process terminates with the homology of the original space.

Spectral sequences were originally introduced by Jean Leray in conjunction with sheaves
in the early 1940s as a prisoner of war in Austria \cite{andler2006jean, dieudonne}.
Leray was interested cohomology, but we are interested
in homology so we will work with cosheaves not sheaves.
The rough idea is the following.
Let $f : \Xspace \to \Y$ be a continuous function.
For every point $y \in \Y$, record the $q$-th homology of the fiber $f^{-1}(y)$.
The resulting object is a cosheaf $\Lfunc^f_q$ over $\Y$.
Now sweep these fiber-wise $q$-dimensional homology classes over $p$-dimensional
subspaces of $\Y$ to produce a cosheaf homology group $\Hfunc_p(\Y; \Lfunc^f_q)$.
There is one such group for every pair of dimensions $p$ and $q$.
One may think, naively, that every element of $\Hfunc_p(\Y; \Lfunc^f_q)$
corresponds to an element of $\Hfunc_{p+q}(\Xspace)$.
In fact, cosheaf homology appears on the second level of the Leray spectral sequence
associated to $f$ and there may be many more levels to climb before the total homology
of $\Xspace$ is reached.

Because of their parallel nature, spectral sequences are of interest to computational topologists.
In 2010, Edelsbrunner and Harer  \cite{comp_top} outlined a spectral sequence to compute the persistent
homology of a filtration which was later implemented in \cite{Bauer}.
The idea here is to divide the boundary matrix associated to the filtration into blocks
which can be processed in parallel and then combined in a hierarchical fashion.
Given a filtered space, can the filtration be divided into subfiltrations over which persistence can be
computed in parallel and then combined in a hierarchical fashion to produce a persistence diagram
for the total filtration?
An early attempt to answer this question can be found in \cite{lipsky2011spectral};
see also \cite{Lewis}.
Basu and Parida relate persistent homology groups to spectral sequences \cite{BASU2017119}.
Govc and Skraba used a spectral sequence in developing a notion of an approximate Nerve theorem \cite{Govc:2018:ANT:3282232.3282249}.

In this paper, we describe in full detail the Leray spectral sequence associated to a 
simplicial map $f : \Xspace \to \Y$.
We then relate the Leray spectral sequence associated to $f$ to the Leray cosheaves of $f$.
When $\Y$ is a triangulation of the real line, the Leray cosheaves of $f$ are
well studied by computational topologists under the guise level-set persistent homology \cite{level_set_zigzag}.
We conclude by relating the Leray spectral sequence to Reeb spaces \cite{Reeb_spaces, patel_thesis}.

Nothing in this paper is new.
The Leray spectral sequence is now a special case of the Grothendieck spectral
sequence which can be found in most textbooks on homological algebra; see for example \cite{weibel}.
The conclusions we draw about level set persistent homology (Corollary \ref{cor:levelset})
are well known although proved in a different language \cite{bendich2013, level_set_zigzag}.
The conclusions we draw about Reeb spaces (Corollary \ref{cor:Reeb}) have also been observed \cite{basu2018reeb, dey}.
The connection between level-set persistent homology and cosheaves was made explicit by Curry \cite{curry} 
which contains a statement similar to our Corollary \ref{cor:levelset} but without a proof.

The task of learning spectral sequences and its connection to sheaves
can be daunting especially in the modern language of derived categories.
Here we choose to develop the Leray spectral sequence in the combinatorial
setting of a simplicial map.
We require the reader has a working knowledge of basic linear algebra and simplicial homology. 
We hope this exposition will be useful to those interested in developing
algorithms based on spectral sequences.

\section{Leray Spectral Sequence}

We now define the Leray spectral sequence associated to a simplicial map
$f : \Xspace \to \Y$.
Throughout this paper, a simplex is always open.

\begin{defn}
\label{defn:spectral}
A \define{spectral sequence} is a collection of three sequences $(\E^r, d^r, \xi^r)_{r\geq 0}$:
\begin{enumerate}
    \item A $\field$-vector space $\E^r$, often called a sheet or page.
    \item A $\field$-linear endomorphism $d^r: \E^r \rightarrow \E^r$ such that $d^r \circ d^r =0$ called a boundary map.
    \item An isomorphism $\xi^{r}: \Hfunc(\E^r) := \dfrac{\ker d^r}{\image\; d^r} \to \E^{r+1}$.
    \end{enumerate}
The spectral sequence is said to \define{converge} if there is some integer $k$ such that $d^r = 0$ 
for all $r\geq k$.
The spectral sequence \define{converges on page $k$} if $d^r = 0$ for all $r \geq k$ and $d^{r-1} \neq 0$.
Note that if $(\E^r, d^r, \xi^r)_{r \geq 0}$ converges on page $k$, then, by Condition 3, 
$\E^r \cong \E^{k}$ for all $r \geq k$.
\end{defn}


We start by setting up the necessary filtrations followed by the construction of the Leray spectral.
We conclude this section with a theorem on convergence and a proposition that will be useful in the next section.

\subsection{Filtrations}

Fix a simplicial map $f : \Xspace \to \Y$ and a field $\field$.
The two complexes $\Xspace$ and $\Y$ need not be finite nor of bounded dimensions.
Consider the simplicial chain complex associated to $\Xspace$:
	\begin{equation}
	\label{eq:chain_complex}
	\begin{tikzcd}
	\cdots \arrow[r] & \Cfunc_{d+1} (\Xspace) \arrow[r, "\partial_{d+1}"] & \Cfunc_{d}(\Xspace)
	\arrow[r, "\partial_d"] & \Cfunc_{d-1}(\Xspace) \arrow[r] & \cdots.
	\end{tikzcd}
	\end{equation}
For any integer $d$, the chain group $\Cfunc_d(\Xspace)$ is the $\field$-vector space generated 
by the set of oriented $d$-simplices in $\Xspace$.
Since $\Xspace$ may be infinite, we require that every element of $\Cfunc_d(\Xspace)$
be non-zero over finitely many $d$-simplices.
In other words, $\Cfunc_d(\Xspace) := \bigoplus_{\sigma \in \Xspace : \dim \sigma = d} \field$.
The boundary homomorphism $\partial$ is induced by the 
operation that takes each oriented $d$-simplex $[u_0, \ldots, u_d ]$ to the signed sum
	\begin{equation*}
	\partial [ u_0, \cdots, u_d] := 
	\sum_{i=0}^{d} (-1)^i [ u_0, \cdots, \hat{u}_i, \cdots, u _d ]
	\end{equation*}
where the hat over $\sigma_i$ indicates its deletion.

The $j$-th-skeleton of $\Y$, denoted $\Y^j$, is the subcomplex of $\Y$ consisting of
all $i$-simplicies such that $i \leq j$.
The filtration of $\Y$ by its skeleta
$\Y^0 \subseteq \Y^1 \subseteq \cdots$
lifts to a filtration
$\Xspace_0 \subseteq \Xspace_1 \subseteq \cdots$
of $\Xspace$ where
$\Xspace_{j}:=\big \{ \sigma \in \Xspace :  f(\sigma) \in \Y^j \big \}$.
For any pair of integers $p$ and $q$, let $\Xspace_{p, q}$ be the set of $(p+q)$-simplices of 
$\Xspace$ lying in $\Xspace_{p}$.
Let $\X_{p, q} \subseteq \Cfunc_{p+q} ( \Xspace )$ be the subspace generated by the set $\Xspace_{p, q}$.
Consider the following diagram:
 \[
  \begin{tikzcd}
  & \mathsf{X}_{p,q} \arrow[r, dashrightarrow, "\partial"] \arrow[hookrightarrow]{d}{} &  \mathsf{X}_{p,q-1}\arrow[hookrightarrow]{d}
  {} & \\
    \cdots \arrow{r}{\partial}& \Cfunc_{p+q}(\Xspace )\arrow{r} {\partial} & \Cfunc_{p+q-1}(\Xspace)
    \arrow{r}{\partial} & \cdots.
  \end{tikzcd}
\]
Every face of a simplex in $\Xspace_{p,q}$ lies in $\Xspace_{p, q-1}$.
As a consequence, the boundary map $\partial : \Cfunc_{p+q} (\Xspace) \to \Cfunc_{p+q-1} (\Xspace)$
restricts to a boundary map
$\partial : \X_{p,q} \to \X_{p,q-1}$.
Consider the following commutative diagram where $r$ is any integer:
\begin{equation}
\label{eq:double_chain}
\begin{tikzcd}
\cdots \arrow[r] & \X_{p-r, q+r+1} \arrow[d, hook] \arrow[r, "\partial"] & 
\X_{p-r, q+r} \arrow[d, hook] \arrow[r, "\partial"]
& \X_{p-r, q+r-1} \arrow[r] \arrow[d, hook] & \cdots  \\
\cdots \arrow[r] & \X_{p,q+1} \arrow[r, "\partial"] \arrow[d, hook] & \X_{p,q} \arrow[r, "\partial"] 
\arrow[d, hook] & \X_{p,q-1} \arrow[r] \arrow[d, hook] & \cdots \\
\cdots \arrow[r] & \X_{p+r, q-r+1} \arrow[r, "\partial"] \arrow[d, hook] & \X_{p+r,q-r} \arrow[r, "\partial"] 
\arrow[d, hook] & 
\X_{p+r, q-r-1} \arrow[r] \arrow[d, hook] & \cdots \\
\cdots \arrow[r] & \Cfunc_{p+q+1}(\Xspace) \arrow[r, "\partial"] & \Cfunc_{p+q}(\Xspace) \arrow[r, "\partial"] 
& \Cfunc_{p+q-1}(\Xspace) \arrow[r] & \cdots  .
\end{tikzcd}
\end{equation}
Let $\B^r_{p,q} \subseteq \X_{p,q} \subseteq \Cfunc_{p+q}(\Xspace)$ be the subspace
$\B^r_{p,q}:=\mathsf{X}_{p,q} \cap \partial\left(\mathsf{X}_{p+r,q-r+1}\right)$
and let $\Z^r_{p,q}\subseteq \mathsf{X}_{p,q} \subseteq \Cfunc_{p+q}(\Xspace)$ be the subspace
$\Z^r_{p,q}:= \X_{p,q} \cap \partial^{-1}\left( \X_{p-r,q+r-1} \right)$.
We now state a series of elementary observations about $\B^r_{p,q}$ and $\Z^r_{p,q}$ that
can be easily verified by chasing elements in Diagram \ref{eq:double_chain}:
	\begin{align}
	\B^r_{p,q} &\subseteq \B^{r+1}_{p,q}& \Z^{r+1}_{p,q} &\subseteq \Z^r_{p,q} \label{ob:one} \\
	\Z^r_{p,q} &\subseteq \Z^{r+1}_{p+1,q-1}& \B^r_{p,q} &\subseteq \B^{r-1}_{p+1,q-1} \label{ob:two} \\
	\B^r_{p,q} &\subseteq \Z^r_{p,q} \label{ob:three}
	\end{align}



\subsection{Spectral Sequence}

We start with the pages of the Leray spectral sequence associated to $f : \Xspace \to \Y$.
For all $r \geq 0$, let
\[\E^r_{p,q} := \dfrac{\Z^r_{p,q}}{\Z^{r-1}_{p-1,q+1}+\B^{r-1}_{p,q}}.\]
Here the $+$ stands for the internal sum of the two subspaces in $\Cfunc_{p+q}(\Xspace)$.
Define the $r$-th page of the spectral sequence $(\E^r, d^r, \xi^r)_{r \geq 0}$ as
$\E^r := \bigoplus_{\forall p,q} \E^r_{p,q}$.

We now begin the construction of the endomorphisms $d^r : \E^r \to \E^r$
by constructing a boundary map 
$d^r_{p,q} : \E^r_{p,q} \to \E^r_{p-r, q+r-1}$ for all $p$ and $q$.

\begin{lem} \label{lem:A}
$\partial \big( \Z^r_{p,q} ) = B^r_{p-r,q+r-1}$.
\end{lem}
\begin{proof}
\begin{align*}
\partial \left (Z^r_{p,q}\right) &=\partial \left(\mathsf{X}_{p,q}\cap \partial^{-1}\left(\mathsf{X}
_{p-r,q+r-1}\right)\right) && \text{by deifinition} \\
&= \mathsf{X}_{p-r,q+r-1}\cap \partial \left(\mathsf{X}_{p,q}\right) && \text{distribute $\partial$} \\
&= B^r_{p-r,q+r-1} && \text{by definition}.
\end{align*}
\end{proof}

\begin{lem} \label{lem:B}
$\partial ( \Z^{r-1}_{p-1, q+1} + \B^{r-1}_{p,q} ) = \B^{r-1}_{p-r,q+r-1}$.
\end{lem}
\begin{proof}
\begin{align*} 
\partial \left(\Z^{r-1}_{p-1,q+1}+\B^{r-1}_{p,q}\right) &= \partial \left (\Z^{r-1}_{p-1,q+1}\right) + 
\partial \left(\B^{r-1}_{p,q}\right) && \text{distribute $\partial$}\\
& = \B^{r-1}_{p-r,q+r-1} + 
\partial \left(\B^{r-1}_{p,q}\right) && \text{by Lemma \ref{lem:A}} \\
& = \B^{r-1}_{p-r,q+r-1} && \text{by $\partial \circ \partial = 0$.}
\end{align*}
\end{proof}

By Lemmas \ref{lem:A} and \ref{lem:B} and the inclusions 
$\B^r_{p-r,q+r-1} \subseteq \Z^r_{p-r, q+r-1}$ (Observation \ref{ob:three}) and
$\B^{r-1}_{p-r,q+r-1} \subseteq \Z^{r-1}_{p-r-1,q+r} + \B^{r-1}_{p-r,q+r-1}$,
the boundary map
$\partial : \Cfunc_{p+q}(\Xspace) \to \Cfunc_{p+q-1}(\Xspace)$
restricts to a boundary map $d^r_{p,q} : \E^r_{p,q}\rightarrow \E^r_{p-r,q+r-1}$.
Define the endomorphism $d^r : \E^r \to \E^r$ as $d^r := \bigoplus_{\forall p,q} d^r_{p,q}$.
We have $d^r \circ d^r = 0$ because $d^r_{p-r,q+r-1} \circ d^r_{p,q} = 0$.
See Figures \ref{fig:page_zero}, \ref{fig:page_one}, and \ref{fig:page_two} for an illustration
of the first three pages.

Now we begin the task of constructing the isomorphisms $\xi^r$.

\begin{lem}
\label{lem:kernel}
$\ker d^r_{p,q} = \dfrac{\Z^{r+1}_{p,q}}{\Z^{r-1}_{p-1,q+1}+\B^{r-1}_{p,q}}$.
\end{lem}
\begin{proof}
The kernel is 
\begin{align*}
\ker d^r_{p,q} &= \dfrac{\Z^r_{p,q} \cap \partial^{-1} \big( \Z^{r-1}_{p-r-1,q+r} + \B^{r-1}_{p-r,q+r-1} \big)}
{\Z^{r-1}_{p-1,q+1}+\B^{r-1}_{p,q}} \\
&= 
\dfrac{\Z^r_{p,q} \cap \partial^{-1} ( \Z^{r-1}_{p-r-1,q+r}) + \Z^r_{p,q} \cap \partial^{-1}(\B^{r-1}_{p-r,q+r-1})}
{\Z^{r-1}_{p-1,q+1}+\B^{r-1}_{p,q}}.
\end{align*}
We now work on numerator:
	\begin{align*}
	\Z^r_{p,q} \cap \partial^{-1} ( \Z^{r-1}_{p-r-1,q+r}) &= 
	\Z^r_{p,q} \cap \partial^{-1} \big( \X_{p-r-1,q+r} \cap \partial^{-1} ( \X_{p-2r, q+2r-2} ) \big) && \text{by definition} \\
	&= \Z^r_{p,q} \cap \big( \partial^{-1}(\X_{p-r-1, q+r}) \cap \partial^{-1} \partial^{-1} (\X_{p-2r, q+2r-2}) \big) 
	&& \text{distribute $\partial^{-1}$} \\
	&= \Z^r_{p,q} \cap \big( \partial^{-1}(\X_{p-r-1, q+r}) \cap \X_{p-2r,q+2r} \big) \\
	&= \Z^r_{p,q} \cap \partial^{-1}(\X_{p-r-1, q+r}) \\
	\Z^r_{p,q} \cap \partial^{-1}(\B^{r-1}_{p-r,q+r-1}) &= 
	\Z^r_{p,q} \cap \partial^{-1} \big( \X_{p-r,q+r-1} \cap \partial ( \X_{p-1, q+1} ) \big) && \text{by definition} \\
	&= \Z^r_{p,q} \cap \big( \partial^{-1}(\X_{p-r,q+r-1}) \cap \X_{p-1, q+1} \big) && \text{distribute $\partial^{-1}$} \\
	&= \Z^r_{p,q} \cap \Z^{r-1}_{p-1,q+1} && \text{by definition}.
	\end{align*}
Furthermore
\begin{align*}
\Z^r_{p,q} \cap \partial^{-1} \big( \X_{p-r-1,q+r} \big) + \Z^r_{p,q} \cap \Z^{r-1}_{p-1,q+1} &= 
\Z^{r+1}_{p,q} + \Z^r_{p,q} \cap \Z^{r-1}_{p-1,q+1} &&  \\
&= \Z^{r+1}_{p,q} + \Z^{r-1}_{p-1,q+1} && \text{Observation \ref{ob:two}} \\
&= \Z^{r+1}_{p,q} && \text{Observation \ref{ob:two}}.
\end{align*}
\end{proof}

\begin{lem}
\label{lem:image}
$\image\; d^r_{p+r,q-r+1} = \dfrac{\Z^r_{p-1,q+1} + \B^r_{p,q}}{\Z^{r-1}_{p-1,q+1}+\B^{r-1}_{p,q}}$.
\end{lem}
\begin{proof}
We have
\begin{align*}
\image\; d^r_{p+r, q-r+1} &= \dfrac{\partial (\Z^r_{p+r,q-r+1})}{\Z^{r-1}_{p-1,q+1}+\B^{r-1}_{p,q}} \\
&= \dfrac{\B^r_{p,q}}{\Z^{r-1}_{p-1,q+1}+\B^{r-1}_{p,q}} && \text{by Lemma \ref{lem:A}} \\
&= \dfrac{\B^r_{p,q} + \Z^{r-1}_{p-1,q+1}+\B^{r-1}_{p,q}}{\Z^{r-1}_{p-1,q+1}+\B^{r-1}_{p,q}} \\
&= \dfrac{\B^r_{p,q} + \Z^{r-1}_{p-1,q+1}}{\Z^{r-1}_{p-1,q+1}+\B^{r-1}_{p,q}} && \text{Observation \ref{ob:one}}.
\end{align*}
\end{proof}

By Lemmas \ref{lem:kernel} and \ref{lem:image}, 
	\[\Hfunc (\E^r_{p,q}) = \frac{\ker d^r_{p,q}}{\image\; d^r_{p+r, q-r+1}} = 
	\frac{\Z^{r+1}_{p,q}}{\Z^{r}_{p-1, q+1} + \B^{r}_{p,q}} = \E^{r+1}_{p,q}.\]
Thus we have an isomorphism $\xi^r_{p,q} : \Hfunc(\E^r_{p,q}) \to \E^{r+1}_{p,q}$.
Let $\xi^r := \bigoplus_{\forall p,q} \xi^r_{p,q}$.
We are finished defining the Leray spectral sequence $(\E^r, d^r, \xi^r)_{r \geq 0}$
associated to the simplicial map $f : \Xspace \to \Y$.
\begin{figure}
\[
  \begin{tikzcd}
  & \vdots && \vdots && \vdots & \\
    \cdots \arrow{r} &\E^0_{p-1,q+1} \arrow{rr}{d^0_{p-1,q+1}}  && \E^0_{p-1,q} \arrow{rr}{d^0_{p-1,q}} && 
    \E^0_{p-1,q-1} \arrow{r} & \cdots \\
     \cdots \arrow{r} & \E^0_{p,q+1} \arrow{rr}{d^0_{p,q+1}}  && \E^0_{p,q} \arrow{rr}{d^0_{p,q}} && 
     \E^0_{p,q-1} \arrow{r} & \cdots \\
     \cdots \arrow{r} &\E^0_{p+1,q+1} \arrow{rr}{d^0_{p+1,q+1}}  && \E^0_{p+1,q} \arrow{rr}{d^0_{p+1,q}} && \E^0_{p+1,q-1} \arrow{r} & \cdots \\
     & \vdots && \vdots && \vdots &
  \end{tikzcd}
\]
\caption{A visualization of $\E^0$.}
\label{fig:page_zero}
\end{figure}
\begin{figure}
\[
  \begin{tikzcd}
  & \vdots & \vdots &\vdots  & \\
    \cdots  &\E^1_{p-1,q+1} \arrow{u}  & \E^1_{p-1,q} \arrow{u} & \E^1_{p-1,q-1} \arrow{u} & \cdots \\
    &&&& \\
     \cdots  & \E^1_{p,q+1} \arrow{uu}{d^1_{p,q+1}}  & \E^1_{p,q} \arrow{uu}{d^1_{p,q}} & 
     \E^1_{p,q-1} \arrow{uu}{d^1_{p,q-1}} & \cdots \\
     &&&& \\
     \cdots  & \E^1_{p+1,q+1} \arrow{uu}{d^1_{p+1,q+1}}  & \E^1_{p+1,q} \arrow{uu}{d^1_{p+1,q}} & 
     \E^1_{p+1,q-1} \arrow{uu}{d^1_{p+1,q-1}} & \cdots \\
     & \vdots \arrow{u} & \vdots \arrow{u} &\vdots \arrow{u} &
  \end{tikzcd}
\]
\caption{A visualization of $\E^1$.}
\label{fig:page_one}
\end{figure}
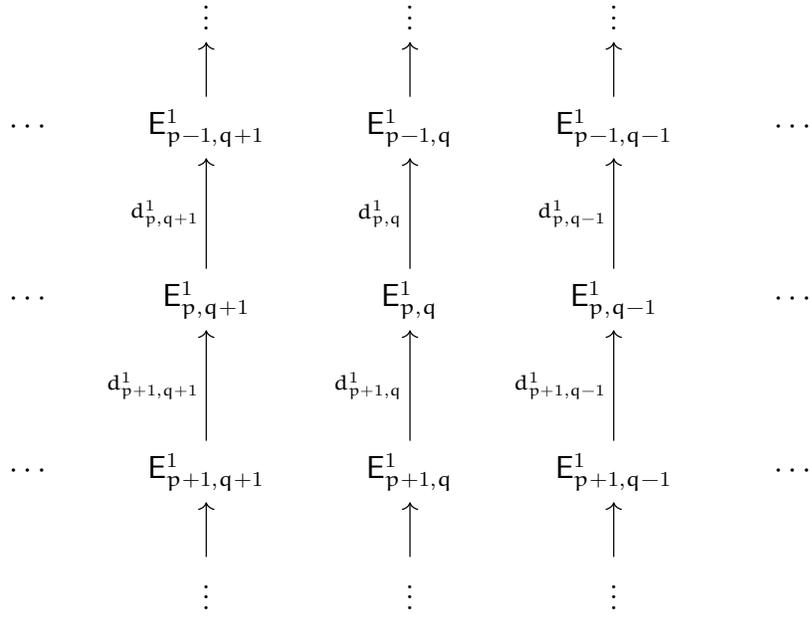

\begin{figure}
\[
  \begin{tikzcd}
 \textrm{ } & \vdots & \vdots &\vdots \\
 \cdots & \E^2_{p-2,q+1} & \E^2_{p-2,q} & \E^2_{p-2,q-1}\\
    \cdots  &\E^2_{p-1,q+1}  & \E^2_{p-1,q}  & \E^2_{p-1,q-1}  & \cdots \\
     \cdots  &\E^2_{p,q+1}   & \E^2_{p,q} \arrow{uul}{d^2} & \E^2_{p,q-1} \arrow{uul}{d^2} & \cdots \\
     \cdots  &\E^2_{p+1,q+1}   & \E^2_{p+1,q} \arrow{uul}{d^2} & \E^2_{p+1,q-1} \arrow{uul}{d^2} & \cdots \\
     \cdots &\E^2_{p+2,q+1}   & \E^2_{p+2,q} \arrow{uul}{d^2} & \E^2_{p+2,q-1} \arrow{uul}{d^2} & \cdots \\
     & \vdots  & \vdots &\vdots
  \end{tikzcd}
\]
\caption{A visualization of $\E^2$ (indices on $d^r$ are dropped for spacing reasons).}
\label{fig:page_two}
\end{figure}
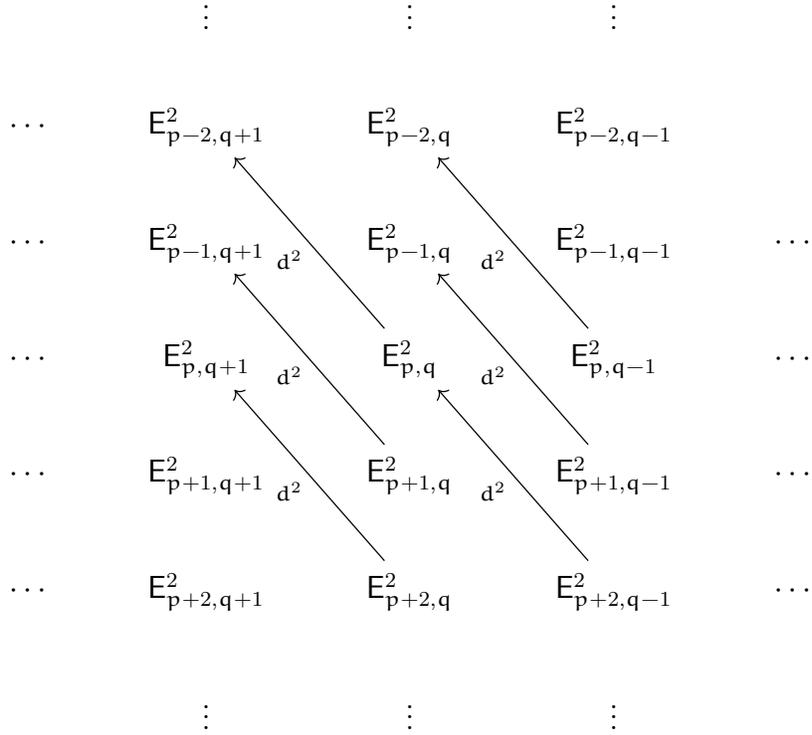

\subsection{Convergence}
If the dimension of $\Y$ is bounded, then the Leray spectral sequence of $f : \Xspace \to \Y$
converges.

\begin{prop}
\label{prop:convergence}
Let $f : \Xspace \to \Y$ be a simplicial map, $(\E^r, d^r, \xi^r)_{r \geq 0}$ its Leray spectral
sequence, and $m = \dim \Y$.
Then $(\E^r, d^r, \xi^r)_{r \geq 0}$ converges on page $r = m+1$.
\end{prop}
\begin{proof}
For $p \leq m$, $\E^r_{p-r, q+r-1} = 0$ because $\X_{p-r,q+r-1} = 0$.
This makes the codomain of $d^r_{p,q}$ zero.
For $p > m$, consider
	\begin{align*}
	\Z^r_{p,q} &= \X_{p,q} \cap \partial^{-1}(\X_{p-r, q+r-1}) \\
	&= \Cfunc_{p+q}(\Xspace) \cap \partial^{-1}(\X_{p-r, q+r-1}) \\
	\Z^{r-1}_{p-1, q+1} &= \X_{p-1, q+1} \cap \partial^{-1}(\X_{p-r, q+r-1}) \\
	&= \Cfunc_{p+q}(\Xspace) \cap \partial^{-1}(\X_{p-r, q+r+1}).
	\end{align*}
This makes the domain of $d^r_{p,q}$ zero.
\end{proof}

\begin{thm}
\label{thm:conv}
Let $f : \Xspace \to \Y$ be a simplicial map, 
$\left(\E^r, d^r, \xi^r \right)_{r \geq 0}$ its Leray spectral sequence, and $m = \dim \Y$.
Then there is a canonical isomorphism
$\Hfunc_k \left(\Xspace \right) \cong \bigoplus_{p=0}^k \E^{m+1}_{p,k-p}$.
\end{thm}
\begin{proof}
Let $\Z_{p+q} := \ker \partial : \Cfunc_{p+q}(\Xspace) \to \Cfunc_{p+q-1}(\Xspace)$
and let $\B_{p+q} := \image \; \partial : \Cfunc_{p+q+1}(\Xspace) \to \Cfunc_{p+q}(\Xspace)$.
Then $\Hfunc_{p+q}(\Xspace) := \dfrac{\Z_{p+q}}{\B_{p+q}}$.
Let $\Z_{p,q} := \X_{p,q} \cap \Z_{p+q}$ and let $\B_{p,q} := \X_{p,q} \cap \B_{p+q}$.
Diagram \ref{eq:double_chain} induces the following diagram of inclusions:
	\begin{equation*}
	\begin{tikzcd}
	\B_{0, p+q} \ar[r, hook] \ar[d, hook] & \cdots \ar[r, hook] & \B_{p-1, q+1} \ar[r, hook] \ar[d, hook]
	& \B_{p, q} \ar[r, hook] \ar[d, hook] & \cdots \ar[r, hook] & \B_{p+q} \ar[d, hook] \\
	\Z_{0, p+q} \ar[r, hook] & \cdots \ar[r, hook] & \Z_{p-1, q+1} \ar[r, hook] & \Z_{p, q} \ar[r, hook] 
	& \cdots \ar[r, hook] & \Z_{p+q}.
	\end{tikzcd}
	\end{equation*}
We have
	\begin{align*}
	\E^{m+1}_{p,q} &= \dfrac{\Z^{m+1}_{p,q}}{\Z^{m}_{p-1, q+1} + \B^{m}_{p,q}} \\
	&= \dfrac{\X_{p,q} \cap \partial^{-1}(\X_{p-m-1, q+m})}
	{\X_{p-1,q+1} \cap \partial^{-1}(\X_{p-1-m, q+m}) + \X_{p,q} \cap \partial(\X_{p+m, q-m+1})} &&
	\text{by definition} \\
	&= \dfrac{\X_{p,q} \cap \Z_{p+q}}{\X_{p-1,q+1} \cap \Z_{p+q} + \X_{p,q} \cap \B_{p+q}} \\
	&= \dfrac{\Z_{p,q}}{\Z_{p-1,q+1} + \B_{p,q}} \\
	&= \dfrac{ \dfrac{\Z_{p,q}}{\B_{p,q}}} { \dfrac{\Z_{p-1, q+1}}{\B_{p-1,q+1}}}
	\end{align*}
which is a direct summand of $\Hfunc_{p+q}(\Xspace)$.
Therefore $\Hfunc_k \left(\Xspace \right) \cong \bigoplus_{p=0}^k \E^{m+1}_{p,k-p}$.
\end{proof}

\section{Leray Cosheaves}
Let $f : \Xspace \to \Y$ be a simplicial map.
For each point $p \in |\Y|$, the pre-image $|f^{-1}|(p)$ is a space which we encode, algebraically,
as a chain complex.
The collection of such chain complexes over all points in $|\Y|$ assembles into a
cosheaf of chain complexes over $\Y$.
Since the data of $f$ is discrete, this cosheaf is discrete: one need only define it on simplices
and face relations of $\Y$.
The \emph{$d$-th Leray cosheaf} of $f$ is the $d$-th homology functor applied pointwise
to this cosheaf of chain complexes.

\subsection{Cosheaf of Chain Complexes}

Fix a simplicial map $f : \Xspace \to \Y$.
For all (open) $\tau \in \Y$ and integers $q$, let 
\[\Xspace_q^f(\tau) := \left \{ \sigma \in \Xspace \; \middle | \; \dim \sigma = \dim \tau + q\text{ and } 
\sigma \in f^{-1}(\tau) \right \}.\]
Note that $\Xspace_q^f(\tau)$ may not be a simplicial complex.
Let $\Ffunc_q (\tau) \subseteq \Cfunc_{\dim \tau+q}(\Xspace)$ be the subspace 
generated by the set $\Xspace_q^f(\tau)$.
The boundary of an oriented simplex $\sigma =  [ u_0, \ldots, u_{\dim \tau+q}  ]$
in $\Xspace_q^f(\tau)$ is 
\[\partial \sigma := \sum_{i=0}^{ \dim \tau + q} (-1)^i
 [ u_0 , \cdots, \hat{u}_i, \cdots, u_{\dim \tau + q}  ]\]
where the hat over $u_i$ indicates its deletion.
Not every simplex in the boundary of $\sigma$ belongs to $\Xspace^f_{q-1}(\tau)$.
Nonetheless, this simplex-wise boundary operation taking each $\sigma \in \Xspace_q^f(\tau)$
to a signed up of simplices that are actually in $\Xspace_{q-1}^f(\tau)$ induces a boundary 
homomorphism 
$\partial : \Ffunc_q(\tau) \to \Ffunc_{q-1}(\tau).$
In this way, $f$ induces a chain complex $\Ffunc_{\bullet}(\tau)$ of $\field$-vector spaces
over each simplex $\tau \in \Y$.
Note that $\Ffunc_\bullet(\tau)$ is simply the restriction of the simplicial chain complex for $\Xspace$
(see Equation \ref{eq:chain_complex}) to $\Xspace_q^f(\tau)$.

Now we define a chain map $\Ffunc_\bullet(\tau' \leq \tau) : \Ffunc_\bullet(\tau) \to \Ffunc_\bullet(\tau')$
for every face relation $\tau' \leq \tau$ in $\Y$
Every simplex $\sigma \in f^{-1}(\tau)$ has a maximal face $\sigma' \leq \sigma$ such that $f(\sigma') = \tau'$.
If $\dim \sigma - \dim \tau = \dim \sigma' - \dim \tau'$, then we send
$\sigma$ to $\sigma'$ with the appropriate sign.
Otherwise, we send $\sigma$ to $0$.
We now work on the appropriate sign for the first case.
Orient $\sigma = [u_0, \cdots, u_{\dim \tau + q}]$ as follows. 
Call a vertex $u_i$ of $\sigma$ \emph{alone} if $u_i$ is the only vertex of $\sigma$ that maps to $f(u_i)$.
Choose the orientation on $\sigma$ so that if $u_i$ is alone, then $u_j$ is also alone for 
all $j \geq i$.
That is, push all the alone vertices of $\sigma$ to the end of the ordering.
Since we are assuming $\dim \sigma - \dim \tau = \dim \sigma' - \dim \tau'$,
there are exactly $\dim \tau - \dim \tau'$ vertices of $\sigma$ that are removed to get $\sigma'$ and
furthermore all these vertices are alone.
Let $\Ffunc_q(\tau' \leq \tau)$ be the map generated by sending $[u_0, \cdots, u_{\dim \tau + q}]$
to the oriented simplex $\sigma'$ obtained by deleting the appropriate vertices.
We now argue commutativity of the following diagram:
	\begin{equation*}
	\begin{tikzcd}
	\cdots \arrow[r, "\partial"]  & 
	\Ffunc_{q+1}(\tau) \arrow[r, "\partial"] \arrow[d, "\Ffunc_{q+1}(\tau' \leq \tau)"] & 
	\Ffunc_{q}(\tau) \arrow[r, "\partial"] \arrow[d, "\Ffunc_q(\tau' \leq \tau)"] &
	\Ffunc_{q-1}(\tau) \arrow[r, "\partial"] \arrow[d, "\Ffunc_{q-1}(\tau' \leq \tau)"] & 
	\cdots \\
	\cdots \arrow[r, "\partial"] & \Ffunc_{q+1}(\tau') \arrow[r, "\partial"] & 
	\Ffunc_{q}(\tau') \arrow[r, "\partial"] & \Ffunc_{q-1}(\tau') \arrow[r, "\partial"] & \cdots.
	\end{tikzcd}	
	\end{equation*}
Suppose $(-1)^i[u_0, \cdots, \hat{u}_i, \cdots , u_{\dim \tau + q}]$ is a term in the signed sum
$\partial [u_0, \cdots u_{\dim \tau + q}]$.
Assuming this simplex belongs to $f^{-1}(\tau)$, $u_i$ cannot be alone.
Note that all the alone vertices of $[u_0, \cdots, \hat{u}_i, \cdots , u_{\dim \tau + q}]$ are at
the end of the ordering.
Then $\Ffunc_{q-1}(\tau' \leq \tau)$ applied to 
$[u_0, \cdots, \hat{u}_i, \cdots , u_{\dim \tau + q}]$
is obtained by deleting the appropriate alone vertices.
Now consider the other direction.
The map $\Ffunc_q(\tau' \leq \tau)$ takes $[u_0, \cdots , u_{\dim \tau + q}]$ to the oriented
simplex obtained by deleting the appropriate alone vertices.
Then $\partial$ applied to the resulting oriented simplex is obtained by deleting
the $i$-th vertex with a sign of $(-1)^i$.
Note that this index $i$ is the same as the $i$ in $[u_0, \cdots, {u}_i, \cdots , u_{\dim \tau + q}]$.
The diagram commutes.
As an immediate consequence, for every triple $\tau'' \leq \tau' \leq \tau$, 
$\Ffunc_\bullet(\tau'' \leq \tau) = \Ffunc_\bullet(\tau'' \leq \tau') \circ \Ffunc_\bullet(\tau' \leq \tau)$.

\begin{defn}
The \define{entrance path category} $\Ent(\Y)$ of the simplicial complex $\Y$ is the poset with an object
for each simplex $\tau \in \Y$ and a morphism $\tau \to \tau'$ for every face relation $\tau' \leq \tau$.
\end{defn}

Given $f : \Xspace \to \Y$, the assignment to each $\tau \in \Y$ the chain complex
$\Ffunc_\bullet(\tau)$ and to each $\tau' \leq \tau$ the chain map $\Ffunc_\bullet(\tau' \leq \tau)$
is a functor
$\Ffunc_\bullet : \Ent(\Y) \to \mathsf{Ch}_\bullet (\mathsf{Vec})$
where $\mathsf{Ch}_\bullet (\mathsf{Vec})$ is the category of chain complexes of $\field$-vector spaces.

\begin{defn}
A \define{cosheaf of vector spaces} over a simplicial complex $\Y$ is a functor 
$\Lfunc : \Ent(\Y) \to \Vect$.
The \define{$q$-th Leray cosheaf} of a simpicial map $f : \Xspace \to \Y$ is the functor
$\Lfunc_q^f := \Hfunc_q \circ \Ffunc_\bullet : \Ent(\Y) \to \Vect$
where $\Hfunc_q : \mathsf{Ch}_\bullet \to \Vect$ is the $q$-th homology functor.
\end{defn}

\subsection{Cosheaf homology}
Fix a cosheaf $\Lfunc : \Ent(\Y) \to \Vect$ of vector spaces over $\Y$.
In this section, we define the homology of $\Y$ with coefficients in $\Lfunc$ also
known as the cosheaf homology of $\Lfunc$.

For an integer $p$, let 
$\Cfunc_p \big( \Y; \Lfunc \big) := 
\bigoplus_{ \tau \in \Y^p \setminus \Y^{p-1} } \Lfunc(\tau)$.
The boundary map $\partial : \Cfunc_p(\Y; \Lfunc) \to \Cfunc_{p-1}(\Y; \Lfunc)$ is generated
by the following assignment.
For an oriented $p$-simplex $\tau = [v_0, \ldots, v_{p}]$ in $\Y$ and $a \in \Lfunc(\tau)$,
let 
\[\partial (a) := \sum_{i=0}^{p} (-1)^i 
\Lfunc \big( [ v_0, \ldots, \hat{v}_i, \ldots, v_{p} ] \leq \tau \big)(a) \in 
\Cfunc_{k-1}\big(\Y; \Lfunc(\tau) \big).\]
We have $\partial \circ \partial (a) = 0$ because $\partial \circ \partial (\tau) = 0$
and $\Lfunc$ is a functor (i.e.\ everything commutes).
Thus we have a chain complex $\Cfunc_\bullet(\Y; \Lfunc)$:
	\begin{equation*}
	\begin{tikzcd}
	\cdots \arrow[r] & \Cfunc_{p+1} \big( \Y; \Lfunc \big) \arrow[r, "\partial"] & 
	\Cfunc_{p} \big( \Y; \Lfunc \big) \arrow[r, "\partial"] &
	\Cfunc_{p-1} \big( \Y; \Lfunc \big) \arrow[r, "\partial"] & \cdots.
	\end{tikzcd}
	\end{equation*}

\begin{defn}
The \define{$p$-th homology of $\Y$ with coefficients in a cosheaf $\Lfunc$} over $\Y$ is
$\Hfunc_p \big(\Y; \Lfunc \big) := \Hfunc_p \big( \Cfunc_\bullet (\Y; \Lfunc) \big)$.
\end{defn}

\begin{prop}
\label{prop:cosheaf_homology}
Let $f : \Xspace \to \Y$ be a simplicial map, $(\E^r, d^r, \xi^r)_{r \geq 0}$ its
Leray spectral sequence, and $\Lfunc_q^f$ its $q$-th Leray cosheaf. 
Then there is a canonical isomorphism
$\Hfunc_{p}(\Y; \Lfunc^f_q) \cong \E^2_{p,q}$.
\end{prop}

\begin{proof}
Consider
	\begin{align*}
	\E^0_{p,q} &= \dfrac{\Z^0_{p,q}}{\Z^{-1}_{p-1, q+1} + \B^{-1}_{p,q}} \\
	&= \dfrac{\X_{p,q} \cap \partial^{-1}(\X_{p,q-1})}{\big(\X_{p-1, q+1} \cap \partial^{-1}(\X_{p,q-1})\big) +
	\big( \X_{p,q} \cap \partial(\X_{p-1,q+2}) \big)} && \text{by definition}  \\
	&= \dfrac{\X_{p,q}}{\X_{p-1,q+1} + \big( \X_{p,q} \cap \partial(\X_{p-1,q+2}) \big)} && \\
	&= \dfrac{\X_{p,q}}{\X_{p-1,q+1}} && \text{by containment}.
	\end{align*}
In other words, $\E^0_{p,q}$ is isomorphic to the $k$-vector space generated by the set 
of $(p+q)$-simplices in $f^{-1}(\Y^p \setminus \Y^{p-1})$.
Consider the following diagram:
	\begin{equation*}
	\begin{tikzcd}
	\cdots \ar[r] & \E^0_{p,q+1} \ar[d, "\mu_{q+1}"] \ar[r, "d^0_{p,q+1}"] & \E^0_{p, q} \ar[d, "\mu_{q}"] 
	\ar[r, "d^0_{p,q}"] & \E^0_{p,q-1} \ar[r] \ar[d, "\mu_{q-1}"] & \cdots \\
	\cdots \ar[r] & \bigoplus_{ \tau \in \Y^p \setminus \Y^{p-1}} \Ffunc_{q+1}(\tau) \ar[r, "\partial"] & 
	\bigoplus_{ \tau \in \Y^p \setminus \Y^{p-1}} \Ffunc_{q}(\tau) \ar[r, "\partial"] & 
	\bigoplus_{ \tau \in \Y^p \setminus \Y^{p-1}} \Ffunc_{q-1}(\tau) 
	\ar[r] & \cdots .
	\end{tikzcd}
	\end{equation*}
Each $(p+q)$-simplex $\sigma \in f^{-1}(\Y^p \setminus \Y^{p-1})$ belongs to $\Xspace^f_q(\tau)$
where $\tau = f(\sigma)$.
Let $\mu_q$ be the map generated by sending the oriented simplex $\sigma$ to itself.
The boundary maps $d^0_{p,q}$ and $\partial$ are defined in exactly the same way.
This makes $\mu_\bullet$ a chain complex isomorphism.
As a result, $\E^1_{p,q}$ is canonically isomorphic to $\Cfunc_p(\Y; \Lfunc^f_q)$.

Consider the following diagram where each $\nu_\ast$ is the isomorphism just constructed:
	\begin{equation*}
	\begin{tikzcd}
	\cdots \ar[r] & \E^1_{p+1,q} \ar[d, "\nu_{p+1}"] \ar[r, "d^1_{p+1,q}"] & \E^1_{p, q} \ar[d, "\nu_{p}"] 
	\ar[r, "d^1_{p,q}"] & \E^1_{p-1,q} \ar[r] \ar[d, "\nu_{p-1}"] & \cdots \\
	\cdots \ar[r] & \Cfunc_{p+1}(\Y; \Lfunc^f_q) \ar[r, "\partial"] & 
	\Cfunc_p(\Y; \Lfunc^f_q) \ar[r, "\partial"] & 
	\Cfunc_{p-1}(\Y; \Lfunc^f_q) 
	\ar[r] & \cdots .
	\end{tikzcd}
	\end{equation*}
Let $\tau \in \Y$ be a $p$-simplex and $\sigma \in f^{-1}(\tau)$ a $(p+q)$-simplex.
Let $\tau$ be the oriented simplex $[v_0, \cdots, v_p]$ and choose an orientation $[u_0, \cdots, u_{p+q}]$
on $\sigma$ so that if $u_i$ is alone, then $i$ is equal to the index of the vertex $f(u_i)$ of $\tau$.
The map $d^1_{p,q}$ is generated by sending the oriented simplex $\sigma$ to a signed sum of 
codimension-$1$ faces each
belonging to $f^{-1}(\tau')$ for some codimension-$1$ face $\tau' \leq \tau$.
If $[u_0, \cdots, \hat{u}_i, \cdots, u_{p+q}]$ is such a codimension-$1$ face, then its sign is $(-1)^i$.
Recall $\nu_\bullet$ is generated by sending each oriented simplex to itself.
The map $\partial$ is generated by sending the oriented simplex $\sigma$ to a signed sum of codimension-$1$ faces
each belonging to $f^{-1}(\tau')$ for some codimension-$1$ face $\tau' \leq \tau$.
By definition of cosheaf homology, the sign here is $(-1)^j$ where $j$ is the index of the vertex in 
$[v_0, \cdots, v_p]$ that is deleted to get $\tau'$.
If $[u_0, \cdots, \hat{u}_i, \cdots, u_{p+q}]$ is such a codimension-$1$ face of $\sigma$ over $\tau'$, 
then its sign is also $(-1)^i$ because the index of $f(u_i)$ is equal to $i$.
This makes $\nu_\bullet$ a chain complex isomorphism.
As a result, $\E^2_{p,q}$ is canonically isomorphic to $\Hfunc_p(\Y; \Lfunc^f_q)$.

\end{proof}

%
%

\section{Applications}
We now relate the Leray spectral sequence and the Leray cosheaves 
of a simplicial map
to level set persistent homology and Reeb spaces.

\subsection{Level Set Persistence}
We now consider simplicial maps $f : \Xspace \to \Y$ where $\Y$ is a 
triangulation of the real line $\Rspace$.
In this case, the $q$-th Leray cosheaf $\Lfunc^f_q$ of $f$ is the equivalent 
to the $q$-th ``level set persistence module'' of $f$.
In this special case, the Leray cosheaves of $f$ encode the homology of $\Xspace$.
The following claim is an immediate consequence of Theorem \ref{thm:conv}.

\begin{cor}
\label{cor:levelset}
Let $\Y$ be a triangulation of the real line and $f : \Xspace \to \Y$ a simplicial map.
Then $\Hfunc_{k}(\Xspace) \cong \bigoplus_{p+q = k} \Hfunc_{p}(\Y; \Lfunc_q^f)$.
\end{cor}

Those who have studied level set persistent homology know that the homology of $\Xspace$ 
can be read from the ``barcode'' of the Leray cosheaves of $f$.
Assume $\Xspace$ is a finite simplicial complex.
Then each Leray cosheaf $\Lfunc_q^f$ is a finite direct sum of indecomposable cosheaves.
Each indecomposable cosheaf or ``bar'' is one of the following four types where $a \leq b$ are any two vertices of $\Y$:
\begin{itemize}
\item Let $\Bfunc_{[a,b]} : \Ent(\Y) \to \Vect$ be the cosheaf that assigns $\field$ to 
all vertices and edges within the closed interval $[a,b]$ and the identity map to all morphisms
between them. The cosheaf is $0$ elsewhere.
\item Let $\Bfunc_{[a,b)} : \Ent(\Y) \to \Vect$ be the cosheaf that assigns $\field$ to 
all vertices and edges within the half-closed interval $[a,b)$ and the identity map to all morphisms
between them. The cosheaf is $0$ elsewhere.
\item Let $\Bfunc_{(a,b)} : \Ent(\Y) \to \Vect$ be the cosheaf that assigns $\field$ to 
all vertices and edges within the open interval $(a,b)$ and the identity map to all morphisms
between them. The cosheaf is $0$ elsewhere.
\item Let $\Bfunc_{(a,b]} : \Ent(\Y) \to \Vect$ be the cosheaf that assigns $\field$ to 
all vertices and edges within the half-closed interval $(a,b]$ and the identity map to all morphisms
between them. The cosheaf is $0$ elsewhere.
\end{itemize}
All but two of these indecomposable cosheaves have non-zero cosheaf homology:
\begin{equation*}
\begin{aligned}
\Hfunc_0(\Y; \Bfunc_{[a,b]}) \cong \field && \Hfunc_0(\Y; \Bfunc_{[a,b)}) \cong 0 && \Hfunc_0(\Y; \Bfunc_{(a,b)}) \cong 0 && 
\Hfunc_0(\Y; \Bfunc_{(a,b]}) \cong 0 \\ 
\Hfunc_1(\Y; \Bfunc_{[a,b]}) \cong 0 && \Hfunc_1(\Y; \Bfunc_{[a,b)}) \cong 0 && \Hfunc_1(\Y; \Bfunc_{(a,b)}) \cong \field && \Hfunc_1(\Y; \Bfunc_{(a,b]}) \cong 0.
\end{aligned}
\end{equation*}
This means that only two types of bars in the barcode for $\Lfunc_q^f$ contribute to the homology of $\Xspace$.

\subsection{Reeb Spaces}

Let $f : \Xspace \to \Y$ be a simplicial map.
Declare two simplicies $\sigma, \sigma' \in \Xspace$ as \emph{equivalent}, denoted $\sigma \sim \sigma'$,
if $f(\sigma) = f(\sigma')= \tau$ and both belong to the same component of $f^{-1}(\tau)$.
The quotient $\Xspace^f := \Xspace/ \sim$ is a simplicial complex and $f$ induces a simplicial map
$\tilde f : \Xspace^f \to \Y$ where $\tilde f (\tilde \sigma) := f(\sigma)$ for any $\sigma$ in the equivalence class
represented by $\tilde \sigma$.
The following diagram of simplicial maps commutes:
	\begin{equation*}
	\begin{tikzcd}
	\Xspace \ar[rd, "f"] \ar[rr, twoheadrightarrow] && \Xspace^f \ar[ld, "\tilde f"] \\
	& \Y. &
	\end{tikzcd}
	\end{equation*}
The  map $\tilde f$ is the \define{Reeb space} of $f$.
We can relate the homology of $\Xspace^f$ to the homology of $\Xspace$ by 
studying the Leray cosheaves of $f$ and $\tilde f$.
For any $\tau \in \Y$, there is a one-to-one correspondence between the set of connected components
of $f^{-1}(\sigma)$ and the set of connected components of $\tilde f^{-1}(\sigma)$.
This means that the two Leray cosheaves $\Lfunc^{\tilde f}_0$ and $\Lfunc^f_0$ are canonically isomorphic.

\begin{cor}
\label{cor:Reeb}
Let $f : \Xspace \to \Y$ be a simplicial map and $\tilde f: \Xspace^f \to \Y$ its Reeb space.
Then there is a canonical isomorphism $\Hfunc_0(\Xspace^f) \cong \Hfunc_0(\Xspace)$
and a canonical inclusion $\Hfunc_1(\Xspace^f) \hookrightarrow \Hfunc_1(\Xspace)$.
\end{cor}
\begin{proof}
Suppose $m = \dim \Y$.
Let $(\E^r, d^r, \xi^r)_{r \geq 0}$ be the Leray spectral sequence for $f$.
Let $(\D^r, e^r, \eta^r)_{r \geq 0}$ be the Leray spectral sequence for $\tilde f$.
By Theorem \ref{thm:conv}, $\Hfunc_0(\Xspace)$ is canonically isomorphic to $\E^{m+1}_{0,0}$
and $\Hfunc_1(\Xspace)$ is canonically isomorphic to $\E^{m+1}_{0,1} \oplus \E^{m+1}_{1,0}$.
Also by Theorem \ref{thm:conv}, $\Hfunc_0(\Xspace^f)$ is canonically isomorphic to
$\D^{m+1}_{0,0}$ and $\Hfunc_1(\Xspace^f)$ is canonically isomorphic to 
$\D^{m+1}_{0,1} \oplus \D^{m+1}_{1,0}$.

By Proposition \ref{prop:cosheaf_homology}, $\E^2_{0,0}$ is canonically isomorphic
to $\Hfunc_0(\Y; \Lfunc^f_0)$ and $\D^2_{0,0}$ is canonically isomorphic to
$\Hfunc_0(\Y; \Lfunc^{\tilde f}_0)$.
Since $\Lfunc^{\tilde f}_0$ and $\Lfunc^f_0$ are canonically isomorphic, 
$\Hfunc_0(\Y; \Lfunc^f_0)$ is canonically isomorphic to $\Hfunc_0(\Y; \Lfunc^{\tilde f}_0)$.
We have our first claim.

The boundary maps $d^r_{0,1} : \E^r_{0,1} \to \E^r_{-r, r}$
and $d^r_{1,0} : \E^r_{1,0} \to \E^r_{1-r, r-1}$, for all $r \geq 2$,
are $0$ because both $\E^r_{-r, r}$ and $\E^r_{1-r, r-1}$ are zero.
As a result, $\E^{m+1}_{0,1} \oplus \E^{m+1}_{1,0}$ is canonically isomorphic
to $\E^2_{0,1} \oplus \E^2_{1,0}$.
Similarly, the boundary map $a^r : \D^r_{1,0} \to \D^r_{1-r, r-1}$, for all $r \geq 2$,
is zero because $\D^r_{1-r, r-1}$ is zero.
As a result, $\D^{m+1}_{1}$ is canonically isomorphic to $\D^{2}_{1,0}$.
By Proposition \ref{prop:cosheaf_homology}, 
$\E^2_{0,1} \oplus \E^2_{1,0}$ is canonically isomorphic to
$\Hfunc_0(\Y; \Lfunc^f_1) \oplus \Hfunc_1(\Y; \Lfunc^f_0)$ and
$\D^2_{1,0}$ is canonically isomorphic to $\Hfunc_1(\Y; \Lfunc^{\tilde f}_0)$.
Since $\Lfunc^{\tilde f}_0$ and $\Lfunc^f_0$ are isomorphic, 
$\Hfunc_1(\Y; \Lfunc^f_0)$ is isomorphic to $\Hfunc_1(\Y; \Lfunc^{\tilde f}_0)$.
We have our second claim.
\end{proof}

The higher homology groups of $\Xspace^f$ do not generally include into the higher
homology groups of $\Xspace$.
In fact, $\Hfunc_d(\Xspace^f)$ can be arbitrarily bigger than $\Hfunc_d(\Xspace)$ for $d \geq 2$;
see \cite{basu2018reeb}.
Consider the following example.
Let $\Xspace$ and $\Y$ be the simplicial complexes illustrated in Figure \ref{fig:ex}.
Let $f : \Xspace \to \Y$ be the simplicial map that takes $u_0 \mapsto v_0$, $u_1 \mapsto v_1$,
$u_2 \mapsto v_2$, and the rest to $v_3$.
The pre-image of every simplex in $\Y$ is non-empty and connected.
This makes $\Xspace^f$ homeomorphic to the $2$-sphere.
As a result, $\Hfunc_2(\Xspace^f) \cong \field$ but $\Hfunc_2(\Xspace) = 0$.

	\begin{figure}
	\centering
	\includegraphics{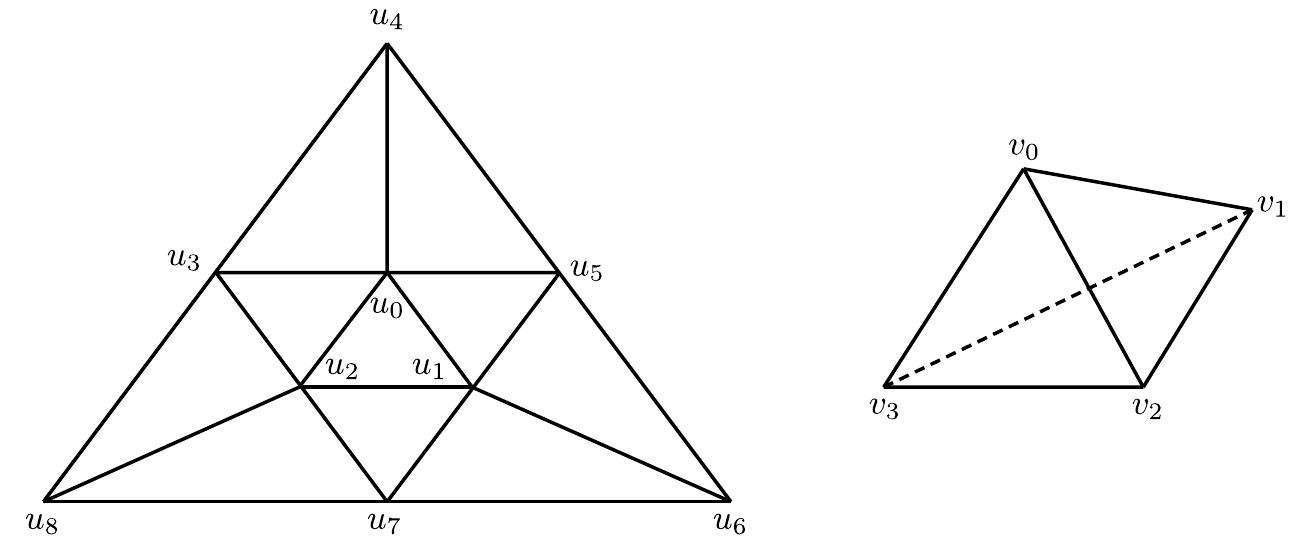}
	\caption{Let $\Xspace$ be the simplicial complex on the left and $\Y$ be the triangulation of the 
	$2$-sphere on the right.}
	\label{fig:ex}
	\end{figure}

\bibliography{sample}

\end{document}